\documentclass{amsart}

\usepackage{amsmath,amsthm}
\usepackage{amsfonts,amssymb}

\newtheorem{thm}{Theorem}[section]

\newtheorem{lem}[thm]{Lemma}

\theoremstyle{remark}

 \def\eb{{\mathbf e}}

 \def\xb{{\mathbf x}}
 \def\yb{{\mathbf y}}
 \def\zb{{\mathbf z}}

 \def\Kb{{\mathbf K}}

 \def\CG{{\mathcal G}}

 \def\PP{{\mathbb P}}
 
 \def\RR{{\mathbb R}}
 \def\SS{{\mathbb S}}

 \def\alb{\mathbf{\alpha}}
 \def\beb{\mathbf{\beta}}

\newcommand{\abs}[1]{\left|#1\right|}

\newcommand{\inti}{\int_{-1}^1}
\newcommand{\intbd}{\int_{B^d}}
\newcommand{\intim}{\int_{I^m}}
\newcommand{\intsd}{\int_{S^{d-1}}}

\newcommand{\jacobi}[2]{P_{#1}^{(\alpha, \beta)}(#2)}
\newcommand{\onjacobi}[2]{p_{#1}^{(\alpha, \beta)}(#2)}

\newcommand{\onjacobimore}[4]{p_{#3}^{(#1,#2)}(#4)}

\begin{document}
\title[]
{Ces\`aro Summability of Fourier Orthogonal Expansions on the Cylinder} 
 
\author{Jeremy Wade}
\address{Department of Mathematics\\ Pittsburg State University\\
    Pittsburg, KS 66762.}
\email{jwade@pittstate.edu}

\date{\today}
\thanks{}

\begin{abstract}
	A result concerning the Ces\`aro summability of the Fourier orthogonal expansion of a function on the cylinder, where the orthogonal basis consists of orthogonal polynomials, in the $L^p$ norms is presented. An upper bound for critical index $\delta$ is obtained.
\end{abstract}

\maketitle

\section{Introduction}
\setcounter{equation}{0}

Summability methods of orthogonal expansions of functions is a classical topic which still receives a considerable amount of attention.  In particular, research into the Ces\`aro summability of Fourier expansions in a multivariable setting is quite active.  Recent studies include investigations in Ces\'aro and Abel summability on the hexagon in \cite{Xu2}, where it was shown that the $(C,1)$ means of the orthogonal expansion converge uniformly.  Work remains in this direction, as the author conjectures that the convergence holds for $(C,\delta)$ means, as long as $\delta>0$.  

The Ces\`aro summability of a Fourier orthogonal expansion in orthogonal polynomials on the parabolic biangle is studied in \cite{CastellFilbirXu}, where the authors establish an upper bound for $\delta$ to ensure the convergence of $(C,\delta)$ means of orthogonal expansions on the parabolic biangle.  The unit ball and sphere are studied in \cite{DaiXu}, where sharp estimates on the values of $\delta$ to ensure the convergence of the $(C,\delta)$ means are given for expansions in $h$-harmonic polynomials on the unit sphere.  These results are then extended to analagous results on the unit ball and simplex.

Investigations into approximation on the cylinder have also received some attention recently.  In \cite{Marchi}, interpolation and cubature on the cylinder $B^2 \times [-1,1]$ is studied, where the authors investigate the Lebesgue constants using approximate Fekete points discrete Leja points, and weakly admissible meshes.  In \cite{Wade}, an approximation technique on the cylinder $B^2 \times [-1,1]$ relating orthogonal polynomials and the Radon transform on parallel disks was investigated, and a sharp estimate on the Lebesgue constant of the operator was obtained.

There are well-known results on the Ces\`aro summability of expansions on the cube \cite{LiXu} and the ball \cite{Xu}. In this paper, we prove an upper bound on the value of $\delta$ to ensure the convergence of the Ces\`aro means of the Fourier orthogonal expansion in orthogonal polynomials of a function defined on the cylinder $B^d \times [-1, 1]^m$.  For our orthogonal basis, we will be using the product of the orthogonal polynomials on $B^d$ with respect to the weight function $w_\mu(\yb)=(1-\|\yb\|^2)^{\mu-1/2}$, and the product Jacobi polynomials, with multi-index paramaters $\alpha$ and $\beta$, with respect to the product weight function $w^{(\alpha,\beta)}(\xb)=\prod_{i=1}^m (1-x_i)^{\alpha_i}(1+x_i)^{\beta_i}$.  Hence our orthogonal basis is orthogonal to the weight function $w(\alpha,\beta,\mu; \xb, \yb)=w^{(\alpha,\beta)}(\xb)w_\mu(\yb)$. While the orthogonal basis has a product structure, this structure does not translate into a product structure of the Ces\`aro means, so that the aforementioned results on the cube and ball do not provide a trivial result for the cylinder.  Our proof of the result requires reducing to problem of convergence on the cylinder to a problem of convergence on the cube $[-1,1]^{m+1}$, and then using the results proven by Xu and Li in \cite{LiXu} on the Ces\`aro summability of functions on the hypercube.  

\section{Background}

\setcounter{equation}{0}
We denote the $d$-dimensional unit ball by $B^d$ and the $m$-dimensional hypercube $[-1,1]^m$ by $I^m$.  We first present the orthogonal polynomial basis used to obtain the Fourier orthogonal decomposition on the cylinder. We will use the notation $n_d$ to denote the dimension of the space of orthogonal polynomials of degree $n$ on either $B^d$ or $I^d$, which is known to satisfy
\begin{displaymath}
	n_d=\binom{n+d-1}{n}.
\end{displaymath}

On the hypercube $I^m$, the product Jacobi polynomials are used to form a basis.  The univariate Jacobi polynomials $\jacobi{n}{x}$ form an orthogonal polynomial basis on $[-1,1]$ with respect to the weight function $w^{(\alpha,\beta)}(x)=(1-x)^\alpha(1+x)^\beta$, with $\alpha,\beta>-1$; that is,
\begin{equation}
	\label{eqn:jacobi}
	\int_{-1}^1 \, \jacobi{n}{x} \jacobi{m}{x}\, w^{(\alpha,\beta)}(x)\, dx = \delta_{m,n}c_{n,\alpha,\beta},
\end{equation}
where 
\begin{displaymath}
	c_{n,\alpha,\beta}= \frac{2^{\alpha+\beta+1}}{2n+\alpha+\beta+1}\frac{\Gamma(n+\alpha+1)\Gamma(n+\beta+1)}{\Gamma(n+1)\Gamma(n+\alpha+\beta+1)},
\end{displaymath}
see \cite{Szego} for more details.  The orthonormal Jacobi polynomials are obtained by scaling the Jacobi polynomials so that the right side of \eqref{eqn:jacobi} is $\delta_{m,n}$; we denote the resulting orthonormal polynomials by $\onjacobi{n}{x}$.  Next, an orthonormal product Jacobi polynomial on $I^m$ is defined by
\begin{equation}
	\label{eqn:productjacobi}
	P_\gamma^{(\alb,\beb)}(\xb) = \onjacobimore{\alpha_1}{\beta_1}{\gamma_1}{x_1}\onjacobimore{\alpha_2}{\beta_2}{\gamma_2}{x_2}\cdots\onjacobimore{\alpha_m}{\beta_m}{\gamma_m}{x_m}\
\end{equation}
where we let $\alpha=(\alpha_1, \alpha_2, \ldots, \alpha_m)$, and similarly define $\beta$, $\gamma$, and $\xb$.  We will consider the degree of $P_\gamma^{(\alb, \beb)}(\xb)$ to be the total degree, $|\gamma|:=\gamma_1+\gamma_2+\cdots+\gamma_m$.  With the weight function $w^{(\alpha,\beta)}(\xb)=\prod_{i=1}^m w^{(\alpha_i,\beta_i)}(x_i)$, an orthonormal basis for the space of polynomials of degree $n$ on $I^m$ is obtained by taking the set of all polynomials of the form in \eqref{eqn:productjacobi}, with $|\gamma|=n$, where orthonormality is in the sense that
\begin{equation}
	\label{eqn:cubeon}
	\int_{I^m}\,P^{(\alpha,\beta)}_\gamma(\xb) P^{(\alpha,\beta)}_\eta(\xb)w^{(\alpha,\beta)}(\xb)\, d\xb =\delta_{\gamma,\eta}.
\end{equation}
It will be convenient to follow the notation developed in Chapter 2 of \cite{DunklXu} and list the orthonormal basis of degree $n$ polynomials on $I^m$ in column vector form.  We define
\begin{displaymath}
	\PP_n^{(\alb,\beb)}(\xb) = \left[ \begin{array} {c}
						\jacobi{\gamma^1}{\xb}\\
						\\
						\jacobi{\gamma^2}{\xb}\\
						\vdots\\
						\\
						\jacobi{\gamma^{n_m}}{\xb}\\
					\end{array}\right],
\end{displaymath}
where $\gamma^1, \gamma^2, \ldots, \gamma^{n_m}$ are multi-indices with $|\gamma|=n$.

The following product formula for Jacobi polynomials, which appears in \cite[p. 262]{Gasper}, will be important in the proof of our result.
\begin{thm}\label{thm:Gasper}
Let $\alpha, \beta > -1$ and $\alpha \geq \beta$.  An integral representation of the form
\begin{displaymath}
	\jacobi{n}{x} \jacobi{n}{y} =
	\int_{-1}^{1}\, \jacobi{n}{1} \jacobi{n}{z} 
	\, K^{(\alpha,\beta)}(x,y,z) \, w^{(\alpha,\beta)}(z)\,dz
\end{displaymath}
exists, with the function $K$ satisfying
\begin{displaymath}
	\int_{-1}^1 \, \left| K^{(\alpha,\beta)}(x,y,z) \right| 
	\, w^{(\alpha,\beta)}(z) \, dz 
	\leq M
\end{displaymath}
for $-1<x,y<1$. 
\end{thm}
This result is for univariate Jacobi polynomials, but easily extends to the product Jacobi polynomials as
\begin{displaymath}
	P_n^{(\alpha,\beta)}(\xb) P_n^{(\alpha,\beta)}(\yb) = \intim P_n^{(\alpha,\beta)}(\eb) P_n^{(\alpha,\beta)}(\zb) \Kb^{(\alpha,\beta)}(\xb,\yb,\zb) w^{(\alpha,\beta)}(\zb) d\zb,	
\end{displaymath}
where $\xb,\yb \in I^m$, and
\begin{displaymath}
	\Kb^{(\alpha,\beta)}(\xb,\yb,\zb) = \prod_{i=1}^m K^{(\alpha_i, \beta_i)}(x_i,y_i,z_i),
\end{displaymath}
and $\Kb^{(\alpha,\beta)}(\cdot,\cdot,\cdot)$ satisfies
\begin{displaymath}
	\intim |\Kb^{(\alpha, \beta)}(\xb,\yb,\zb)|w^{(\alpha,\beta)}(\zb) d\zb \leq M,
\end{displaymath}
where $M$ is a constant given by the product of the constants in the univariate case.  

On the $d$-dimensional unit ball $B^d$, we consider the weight function $w_\mu(\yb)=(1-\|\yb\|^2)^{\mu-1/2}$, with $\mu\geq 0$, and denote a basis of orthonormal polynomials of degree $n=|\alpha|$ on $B^d$ by $s_{\alpha^1}^\mu(\yb),\ s_{\alpha^2}^\mu(\yb),\ldots,s^\mu_{\alpha^{n_d}}(\yb)$, where orthonormailty in a similar manner as \eqref{eqn:cubeon}.  Several examples of specific orthonormal bases exist, see \cite[p. 38]{DunklXu}.  If $d=1$ and $\mu=0$, these polynomials correspond to the Chebyshev polynomials of the first kind, while if $d=2$ and $\mu=1/2$, an orthonormal basis is given by the polynomials
\begin{displaymath}
	2^{n/2-j+1}p_{j}^{(0,n-2j)}(2\|\yb\|^2-1)S_{\beta,n-2j}(\yb),
\end{displaymath}
where $0\leq j \leq n$ and $S_{\beta,n-2j}(\yb)$ is a spherical harmonic on $S^1$, defined by
\begin{displaymath}
	S_{1,n}(\theta)=\frac{1}{\sqrt{\pi}}\sin(n(\pi/2 - \theta))\, ,\, S_{2,n}(\theta)=\frac{1}{\sqrt{\pi}}\cos(n(\pi/2 - \theta)),
\end{displaymath}
with $x=\cos \theta$ and $y=\sin \theta$; see \cite{Muller} for more information on spherical harmonics.  One important property of orthogonal polynomials on $B^d$ with respect to $w^\mu$ is the compact formula introduced in \cite{Xu},
\begin{multline}
	\left[\SS^\mu_n(\yb)\right]^T \SS^\mu_n(\yb')=\frac{n+\mu+\frac{d-1}{2}}{\mu+\frac{d-1}{2}} \frac{ (\Gamma(\mu))^2}{2^{2\mu-1}\Gamma(2\mu)}\\
	\times\int_{-1}^1\, C_n^{(\mu+\frac{d-1}{2})}(\xb\cdot\yb + t \sqrt{1-|\xb|^2}\sqrt{1-|\yb|^2})(1-t^2)^{\mu-1}\, dt,
	\label{eq:XuBallProd}
\end{multline}
for $\mu >0$, and 
\begin{multline}
	\left[\SS^0_n(\yb)\right]^T \SS^0_n(\yb')=\frac{n+\frac{d-1}{2}}{\frac{d-1}{2}} \left[ C_n^{(\frac{d-1}{2})}(\xb\cdot\yb + \sqrt{1-|\xb|^2}\sqrt{1-|\yb|^2})\right.\\
	\left.+C_n^{(\frac{d-1}{2})}(\xb\cdot\yb - \sqrt{1-|\xb|^2}\sqrt{1-|\yb|^2})\right]
	\label{eq:XuBallProd2}
\end{multline}
for $\mu=0$.

	Following the column vector notation used on $I^m$, we will use the notation
	\begin{displaymath}
		\SS^\mu_n(\yb)= \left[
		\begin{array}{c}
			s_{\alpha^1}^\mu(\yb)\\
			\\
			s_{\alpha^2}^\mu(\yb)\\
			\vdots\\
			\\
			s_{\alpha^{n_d}}^\mu(\yb)
		\end{array}
		\right]
	\end{displaymath}
to denote the column vector whose elements form an orthonormal basis of polynomials of degree $n$ on $B^d$ with respect to $w_\mu$.

The reproducing kernel of degree $n$ with respect to a weight function $w$ on some set $X$ of positive Borel measure, $K_n(\xb, \xb')$, is a function satisfying
\begin{displaymath}
	\int_X\, f(\xb) K_n(\xb,\xb')\, w(\xb)d\xb = f(\xb')
\end{displaymath}
when $f$ is a polynomial of degree less than or equal to $n$.  Given a basis of orthonormal polynomials with respect to $w$, $\{P_\alpha(\xb)\}_{\alpha \in A}$, for the space of polynomials of degree less than or equal to $n$, the reproducing kernel has the form
\begin{displaymath}
	K_n(\xb,\xb')=\sum_{\alpha \in A} P_\alpha(\xb) P_\alpha(\xb').
\end{displaymath}
The $n$'th partial sum of the Fourier orthogonal expansion of an integrable function $f$ on $X$ with respect to $w$, $S(w;f)$, is defined by
\begin{displaymath}
	S_n(w;f)(\xb')=\int_X\, f(\xb)K_n(\xb,\xb')\, w(\xb)d\xb.
\end{displaymath}
With the column vector notation above, the reproducing kernel on the hypercube with respect to $w^{(\alpha,\beta)}$ can be written as 
\begin{displaymath}
	K_n(\xb,\xb')=\sum_{k=0}^n \left[\PP_k^{(\alpha,\beta)}(\xb) \right]^T\PP_k^{(\alpha,\beta)}(\xb'),
\end{displaymath}
and a similar formula holds for the reproducing kernel on the ball with respect to $w_\mu$.  We may then write the reproducing kernel on $B^d \times I^m$ with respect to the weight function $w(\alpha,\beta,\mu;\xb,\yb):=w^{(\alpha,\beta)}(\xb)w_\mu(\yb)$, $K_n(\xb,\xb',\yb,\yb')$, as
\begin{displaymath}
	K_n(\xb,\xb',\yb,\yb')=\sum_{k=0}^n \sum_{j=0}^k \left[\PP_j^{(\alpha,\beta)}(\xb) \right]^T\PP_j^{(\alpha,\beta)}(\xb') \left[\SS^\mu_{k-j}(\yb)\right]^T \SS^\mu_{k-j}(\yb').
\end{displaymath}
We define the $n$'th partial sum of the Fourier orthogonal expansion of an integrable function $f$ on $B^d \times I^m$ to be 
\begin{displaymath}
	S_n(\mu,\alpha,\beta;f)(\xb',\yb')=\int_{B^d}\int_{I^m}K_n(\xb,\xb',\yb,\yb')f(\xb,\yb) w(\alpha,\beta,\mu;\xb,\yb)\,d\xb\,d\yb.
\end{displaymath}

Given a series $\sum s_n$, the Ces\`aro means, or $(C,\delta)$ means, of the series is defined to be
\begin{displaymath}
	\sum_{j=0}^n c_{n,j}^\delta \sum_{k=0}^j s_k,
\end{displaymath}
where $c_{n,j}^\delta = \frac{(-n)_j}{(-n-\delta)_j}$, and $(n)_j=\prod_{k=1}^j (n+k-1)$.  If we define
\begin{displaymath}
	K_n^\delta (\xb,\xb',\yb,\yb')=\sum_{k=0}^n\, c_{n,k}^\delta \sum_{j=0}^k \left[\PP_j^{(\alpha,\beta)}(\xb) \right]^T\PP_j^{(\alpha,\beta)}(\xb') \left[\SS^\mu_{k-j}(\yb)\right]^T \SS^\mu_{k-j}(\yb'),
\end{displaymath}
then the Ces\`aro means of of order $\delta$, or the $(C,\delta)$ means, of the Fourier orthogonal expansion of $f$ are defined by
\begin{displaymath}
	S_n^\delta(\mu, \alpha, \beta; f)(\xb',\yb')=\int_{B^d}\int_{I^m}\, K_n^\delta(\xb,\xb',\yb,\yb')f(\xb,\yb)w(\alpha,\beta,\mu;\xb,\yb) d\xb \, d\yb.
\end{displaymath}
We will be investigating the value of $\delta$ for which this series converges in the space $L^p(B^d\times I^m; w(\alpha,\beta\mu;\xb,\yb))$ - that is, for what $\delta$ is
\begin{displaymath}
	\lim_{n \rightarrow \infty} \int_{B^d \times I^m} \, \left|S_n^\delta(\alpha,\beta,\mu;f)(\xb,\yb)-f(\xb,\yb)\right|^p w(\alpha,\beta,\mu;\xb,\yb) d\xb d\yb =0.
\end{displaymath}

\section{Main Theorem}
\setcounter{equation}{0}
We now present our main result for this paper.
\begin{thm}
	\label{thm:main}
	Let $f$ be a continuous function on $B^d \times I^m$, and suppose that $\mu\geq 0$, $\alpha_i>-1$, $\beta_i>-1$, and $\alpha_i+\beta_i \geq -1$ for $1 \leq i \leq m$.  The Ces\`aro means of the Fourier orthogonal expansion of $f$ with respect to $w(\alpha,\beta,\mu;\xb,\yb)$ converge in $L^p(B^d \times I^m; w(\alpha,\beta,\mu;\xb,\yb))$, with $1 \leq p < \infty$, and $C(B^d \times I^m)$, to $f$ if
	\begin{multline}
		\delta > \sum_{i=1}^m \max\{\alpha_i,\beta_i\} + \mu + \frac{d+m-1}{2}\\
		+\max \left\{ 0, -\sum_{i=1}^m\min\left\{\alpha_i,\beta_i\right\} -\mu-\frac{d+m+1}{2}\right\}.
	\end{multline}
\end{thm}

Our proof of Theorem \ref{thm:main} will ultimately reduce the question of convergence on the cylinder to that of the hypercube.  We will show that the integral
\begin{displaymath}
	\intim \intbd \abs{ K_n^\delta (\xb, \xb', \yb, \yb') }
	w(\alpha,\beta,\mu;\xb,\yb) \,d\xb d\yb
\end{displaymath}
is uniformly bounded by some constant $M$ which is independent of $n$, $\xb'$, and $\yb'$, which will then imply summability, by results known on the hypercube.  Throughout our proof, $c$ will denote a positive constant that may change values from line to line.

Following Lemma 2.2 \text{in} \cite{LiXu}, we first show that it is enough to consider $\xb'=\eb:=(1,1,\ldots,1)$ in the hypercube.
\begin{lem}  
In order to prove the convergence of $(C,\delta)$ means of the orthogonal expansion, it suffices to prove 
\begin{equation} 
	\label{eq:cesaromeans}
	\intim \intbd \abs{ K_n^\delta ( \xb, \eb, \yb, \yb') }
	w(\alpha,\beta,\mu;\xb,\yb) \, d\yb \, d\xb \leq M
\end{equation}
for $M$ independent of $n$ and $\yb'$.
\end{lem}
\begin{proof}
Using Theorem \ref{thm:Gasper}, we obtain 
\begin{align*}
	A_n^\delta&:= \intim \intbd \abs{ K_n^\delta(\xb,\xb',\yb,\yb') }
	w(\alpha, \beta, \mu;\xb ,\yb) d\yb \, d\xb\\
	& = \intim \intbd \abs{ \sum_{j=0}^n c_{n,j}^\delta 
	\sum_{k=0}^j \left[ \PP_{j-k}^{(\alpha,\beta)} (\xb) \right]^T
	\PP_{j-k}^{(\alpha,\beta)} (\xb') 
	\left[ \SS_{k}^{\mu} (\yb) \right]^T
	\SS_{k}^\mu(\yb')}\\
	& \phantom{AAAAA} \times w(\alpha,\beta, \mu;,\xb, \yb)\, d\yb d\xb\\
	& \leq \intim \intbd \intim \abs{ \sum_{j=0}^n c_{n,j}^\delta 
	\sum_{k=0}^j \left[ \PP_{j-k}^{(\alpha,\beta)} (\eb) \right]^T
	\PP_{j-k}^{(\alpha,\beta)} (\zb) 
	\left[ \SS_{k}^{\mu}(\yb) \right]^T
	\SS_{k}^\mu (\yb') }\\
	& \phantom{AAAAA} \times \abs{ \Kb (\xb,\xb',\zb) } w^{(\alpha,\beta)} (\zb) d\zb\,
	w(\alpha, \beta, \mu; \xb, \yb)\, d\yb\,d\xb.\\
\end{align*}
Applying Fubini's theorem gives
\begin{align*}
	A_n^\delta&\leq \intbd \intim \abs{ \sum_{j=0}^n c_{n,j}^\delta 
	\sum_{k=0}^j \left[ \PP_{j-k}^{(\alpha,\beta)} (\eb) \right]^T
	\PP_{j-k}^{(\alpha,\beta)} (\zb) 
	\left[ \SS_{k}^{\mu}(\yb) \right]^T
	\SS_{k}^\mu(\yb')}\\
	& \qquad \times \intim \abs{ \Kb (\xb,\xb',\zb) } \, w^{(\alpha,\beta)} (\xb) \,d\xb \, 
	w(\alpha,\beta,\mu;\zb,\yb)\, d\zb \, d\yb\\
	&\leq M \intim \intbd \abs{ K_n^\delta(\eb,\zb,\yb,\yb') }
	w(\alpha,\beta,\mu; \zb,\yb)\, d\yb\, d\zb.\\
\end{align*}
Replacing $\zb$ with $\xb$ and switching the places of $\eb$ and $\xb$ proves the lemma.
\end{proof}

Our next lemma reduces the integral over $B^d$ to an integral over $[-1,1]$ of a Gegenbauer polynomial.  The idea for this lemma comes from the proof of Theorem 5.3 \text{in} \cite{Xu}. We define
\begin{displaymath}
	\CG_{\mu}^{(\alpha,\beta)}(\yb'):= \intbd \abs{ K_n^\delta(\xb,\eb,\yb,\yb') }
	w_\mu(\yb) d\yb
\end{displaymath}
and 
\begin{displaymath}
	F_{n,\mu}^\delta( \cdot ):=\sum_{j=0}^n c_{j,n}^\delta \sum_{k=0}^j 
	\frac{k +\mu +\frac{d-1}{2}}{\mu +\frac{d-1}{2}}
	C_k^{(\mu + \frac{d-1}{2})} ( \cdot ) 
	\left[ \PP_{j-k}^{(\alpha,\beta)} (\xb) \right]^T 
	\PP_{j-k}^{(\alpha,\beta)} (\eb). 
\end{displaymath}

\begin{lem} For $\mu \geq 0$, 
	\begin{equation}
		\label{poopinequality}
		\CG_\mu^{(\alpha, \beta)}(\yb') \leq c  \inti \left| F_n^\delta (u) \right| \left( 1 - u^2 \right)^{ 
	\frac{ d-2 }{2} + \mu}\, du,
	\end{equation}
\end{lem}
\begin{proof}
We first consider the case $\mu >0$.  Substitute \eqref{eq:XuBallProd} into \eqref{eq:cesaromeans} to obtain
\begin{multline} \label{eq:ballest}
	\CG_{\mu}^{(\alpha,\beta)}(\yb')=\intbd \left| \inti F_n^\delta \left( \langle \yb , \yb' \rangle + 
	\sqrt{1-\left| \yb \right|^2} \sqrt{ 1 - \left| \yb' \right|^2 } \,t \right)
	\left( 1 - t^2 \right)^{\mu-1} \, dt \right| \\
	\times w_\mu(\yb)\,  d\yb.
\end{multline}
Applying the change of variable $\yb = r \eta$, where
$\eta \in S^{d-1}$, $0 \leq r \leq 1$, gives
\begin{align*}
	\CG_{\mu}^{(\alpha,\beta)}(\yb')& = \int_0^1 r^{d-1} \intsd  
	\left|  \inti F_n^\delta \left( r \langle \eta , \yb'\rangle 
	+ \sqrt{1 - \left| \yb'\right|^2} \sqrt{1-r^2} \,t \right) 
	\left( 1 - t^2 \right)^{\mu-1} \, dt \right| \\
	& \qquad ( 1-r^2 )^{\mu-1/2} \, d\omega ( \eta ) \, dr, 
\end{align*}
where $d\omega$ is the surface measure on $S^{d-1}$.  Now let $A$ be the rotation matrix satisfying 
$A(\yb') = (0,0,\ldots,0,|\yb'|)$, and apply the change of basis 
$\eta \mapsto A^T \eta$ to obtain
\begin{align*}
	\CG_{\mu}^{(\alpha,\beta)}(\yb') & = \int_0^1 r^{d-1} \intsd \left| 
	\inti F_n^\delta \left( r \eta_d \left| \yb' \right| 
	+ \sqrt{1 - \left| \yb \right|^2} \sqrt{1 - r^2} t \right) 
	\left( 1 - t^2 \right)^{\mu-1} \, dt \right| \\
	& \qquad \times ( 1 - r^2 )^{\mu-1/2} \, 
	d \omega ( \eta ) \, dr
\end{align*}
where $\eta = (\eta_1, \ldots, \eta_d)$.  If we let $\eta_d=s$, then
$\eta = ( \sqrt{ 1 - s^2 } \gamma , s)$ for some $\gamma \in S^{d-2}$, and changing variables gives
\begin{align}
	\label{dammit}
	\CG_{\mu}^{(\alpha,\beta)}(\yb') & = \omega_{d-2}  \int_0^1 r^{d-1} \inti
	\\
	& \notag \times \left| \inti F_n^\delta \left(r s \left| \yb' \right|
	 + \sqrt{ 1 - \left| \yb' \right|^2 } \sqrt{ 1 - r^2 }\,  t \right) 
	\left( 1 - t^2 \right)^{\mu - 1} \, dt  \right| \\
	& \notag \qquad \times ( 1 - r^2)^{ \mu - 1/2} 
	( 1 - s^2)^{\frac{d-3}{2}} \, ds \, dr 
\end{align}
where $\omega_{d-2}$ is the surface area of $S^{d-2}$.  Let $s \mapsto p/r$ so $ds = dp/r$ and move the absolute value inside the innermost integral to obtain
\begin{align*}
	\CG_{\mu}^{(\alpha,\beta)}(\yb')  & \leq \omega_{d-2} \int_0^1 \int_{-r}^r \inti
	 \left|  F_n^\delta \left( p \left| \yb' \right|
	 + \sqrt{ 1 - \left| \yb' \right|^2 } \sqrt{ 1 - r^2 }\, t \right) 
	\left( 1 - t^2 \right)^{\mu - 1}  \right| \, dt \\
	& \qquad \times ( 1 - r^2)^{ \mu - 1/2} r
	( r^2 - p^2)^{\frac{d-3}{2}} \, dp \, dr .
\end{align*}
Switching the order of integration of $r$ and $p$ and applying the change of variable $q \mapsto \sqrt{1-r^2}t$, $dq = \sqrt{1-r^2}\, dt$ gives
\begin{align*}
	\CG_{\mu}^{(\alpha,\beta)}(\yb')  & \leq \omega_{d-2}  \inti \int_{\left| p \right|}^1  
	\int_{-\sqrt{1-r^2}}^{\sqrt{1-r^2}}
	\left|  F_n^\delta \left( p \left| \yb' \right|
	+ \sqrt{ 1 - \left| \yb' \right|^2 } q \right) 
	\right|\\
	&\times \left( 1-r^2 - q^2 \right)^{\mu - 1} \, dq \,r
	( r^2 - p^2)^{\frac{d-3}{2}} \, dr \, dp .
\end{align*}
Switching the order of integration of $q$ and $r$ gives
\begin{align*}
	\CG_{\mu}^{(\alpha,\beta)}(\yb') 	& \leq \omega_{d-2}  \inti 
	\int_{-\sqrt{ 1 -\left| p \right|^2 }}^{ \sqrt{ 1 -\left| p \right|^2} } 
	\left|  F_n^\delta \left( p \left| \yb' \right|
	+ \sqrt{ 1 - \left| \yb' \right|^2 } q \right) \right| \\
	& \qquad \qquad \qquad \qquad
	\times \int_{ \left| p \right| }^{\sqrt{ 1-q^2 }}
	\left( 1-r^2 - q^2 \right)^{\mu - 1} r
	( r^2 - p^2)^{\frac{d-3}{2}} \, dr \, dq \, dp .
\end{align*}
Applying the change of variable $r^2 = u \left( 1- q^2 -p^2\right) + p^2$ shows the inner integral is $\tfrac{1}{2}(1-q^2-p^2)^{\mu+\tfrac{d-3}{2}} B(\mu, \frac{d-1}{2})$, were $B(x,y)$ is the beta function.  Hence, we have the inequality
\begin{align*}
	\CG_{\mu}^{(\alpha,\beta)}(\yb')  & \leq \frac{ \omega_{d-2} B \left( \mu, \frac{ d-1}{2} \right)}{2} \inti 
	\int_{-\sqrt{ 1 -\left| p \right|^2 }}^{ \sqrt{ 1 -\left| p \right|^2} } 
	\left|  F_n^\delta \left( p \left| \yb' \right|
	+ \sqrt{ 1 - \left| \yb' \right|^2 } q \right) \right| \\
	& \qquad \qquad \qquad \qquad \qquad \qquad \qquad
	\times \left( 1 - q^2 - p^2 \right)^{ \frac{d-3}{2} + \mu }\, dq \, dp.\\
\end{align*}
Next, we apply the change of variable $q \mapsto \sqrt{1-p^2}s$ to obtain
\begin{align}
	\label{thebeerequation}
	\CG_{\mu}^{(\alpha,\beta)}(\yb') & \leq c \inti \inti \left|  F_n^\delta \left( p \left| \yb' \right|
	+ \sqrt{ 1 - \left| \yb' \right|^2 } \sqrt{ 1 - p^2 } s \right) \right| \\
	& \notag \qquad \qquad \qquad \qquad \qquad \qquad 
	\times \left( 1 - p^2 \right)^{ \frac{d-2}{2} + \mu }
	\left( 1 - s^2 \right)^{ \frac{d-3}{2} + \mu} \, ds \, dp.
\end{align}
Changing variables once again, we let $ u =  p \left| \yb' \right| + \sqrt{ 1 - \left| \yb' \right|^2 } \sqrt{ 1 - p^2 } s$ 
to obtain
\begin{align}
	\label{thisistoolong}\CG_{\mu}^{(\alpha,\beta)}(\yb')  & \leq c \inti \int_{ p \left| \yb' \right| - 
	\sqrt{1 - \left| \yb' \right|} \sqrt{ 1- p^2 } 
	}^{ p \left| \yb' \right| + 
	\sqrt{1 - \left| \yb' \right|} \sqrt{ 1- p^2 } }
	\left|  F_n^\delta \left( u \right) \right| 
	D_{\frac{d-2}{2}+\mu}(|\yb'|, p, u)\\
	& \notag \qquad \times \left( 1 - p^2 \right)^{ \frac{d-2}{2} + \mu }
	\left( 1 - u^2 \right)^{ \frac{d-2}{2} + \mu} \, du \, dp,
\end{align}
where the function $D_\lambda (v,p,u)$, introduced in \cite{Xu}, 
is defined by
\begin{displaymath}
	D_\lambda \left( v, p, u\right) = \frac{ \left( 1- v^2 - p^2 - u^2 
	+ 2upv \right)^{\lambda - 1/2} }{ \left[ \left( 1-v^2 \right) 
	\left( 1-u^2 \right) \left( 1-p^2 \right) \right]^{ \lambda } }
\end{displaymath}
for $1-v^2 -p^2 -u^2 +2upv \geq 0$ and $0$ otherwise. It is readily verified that
\begin{displaymath}
	\inti D_\lambda \left( u, v, p \right)
	\left( 1 - p^2 \right)^{\lambda} \, dp 
	= 2^{2\lambda}B( \lambda+ 1/2, \lambda + 1/2).
\end{displaymath}
Hence switching the order of integration in \eqref{thisistoolong}, we have
\begin{align}
	\label{morebeer} \CG_{\mu}^{(\alpha,\beta)}(\yb')  & \leq c \inti  \left| F_n^\delta \left( u \right) \right| \inti
	D_{\frac{ d-2 }{ 2 } + \mu} \left( \left| \yb' \right|, u, p \right)
	\left( 1- p^2 \right)^{ \frac{ d-2 }{2} + \mu } \, dp \\
	& \notag \qquad \times \left( 1- u^2 \right)^{ \frac{ d-2 }{2} + \mu } \, du \\
	& \notag \leq c \inti \left| F_n^\delta (u) \right| \left( 1 - u^2 \right)^{ 
	\frac{ d-2 }{2} + \mu}\, du.
\end{align}
This proves the lemma for $\mu > 0$.

Turning our attention now to the case when $\mu = 0$, we substitute \eqref{eq:XuBallProd2} into the left side of \eqref{eq:cesaromeans} and ignore the integral over $I^m$ as before to obtain
\begin{align*}
	 \CG_0^{(\alpha,\beta)}(\yb'):=
	\int_{B^d} & \bigg| 
	F_n^\delta \left( \langle \yb , \yb' \rangle + \sqrt{1-\left| \yb \right|^2} \sqrt{ 1 - \left| \yb' \right|^2 }  \right)\\
	& + F_n^\delta \left( \langle \yb , \yb' \rangle  \sqrt{1-\left| \yb \right|^2} \sqrt{ 1 - \left| \yb' \right|^2 }  \right)
	\bigg| w_0(\yb) d\yb.
\end{align*}
We perform the same change of variables from the case when $\mu >0$ to obtain the equivalent of $\eqref{dammit}$,
\begin{align*}
	\CG_0^{(\alpha,\beta)}(\yb') = & \omega_{d-2}  \int_0^1 r^{d-1} \int_{-1}^1  \bigg| 
	F_n^\delta \left(rs|\yb'| + \sqrt{1-|\yb'|^2} \sqrt{1-r^2} \right) \\
	& + F_n^\delta \left( rs|\yb'| - \sqrt{1-|\yb'|^2} \sqrt{1-r^2} \right) \bigg| 
	(1-r^2)^{-1/2} (1-s^2)^{\tfrac{d-3}{2}}\, ds \, dr.
\end{align*}
Now we substitute $p = \sqrt{1-r^2}$ and let $v = \sqrt{1-|\yb'|^2}$ to obtain
\begin{align}
	\notag
	\CG_0^{(\alpha,\beta)}(\yb') 
	& = \omega_{d-2} \int_{-1}^1 \int_{0}^1 
	\bigg| 
	F_n^\delta \left(\sqrt{1-p^2} \sqrt{1-v^2}s + pv \right) \\
	\notag
	& + F_n^\delta \left( \sqrt{1-p^2} \sqrt{1-v^2}s - pv \right) 
	\bigg| 
	(1-p^2)^{\tfrac{d-2}{2}} (1-s^2)^{\tfrac{d-3}{2}}\, dp\, ds\\
	\notag
	& =\omega_{d-2} \int_{-1}^1 \int_{-1}^1 \left| 
	F_n^\delta \left(\sqrt{1-p^2} \sqrt{1-v^2}s + pv \right) \right| 
	\notag \\
	& \qquad \times (1-p^2)^{\tfrac{d-2}{2}} (1-s^2)^{\tfrac{d-3}{2}}\, dp\, ds .
	\label{wantsabeer}
\end{align}
The right side of \eqref{wantsabeer} is the right side of \eqref{thebeerequation}, with $v$ in place of $|\yb'|$.  Following the same steps of the proof for $\mu > 0$, we obtain the equivalent of \eqref{morebeer},
\begin{align}
	\CG_{0}^{(\alpha,\beta)}(y')  & 
	\leq c \inti \left| F_n^\delta (u) \right| \left( 1 - u^2 \right)^{ 
	\frac{ d-2 }{2}}\, du ,
	\notag
\end{align}
which proves the case for $\mu = 0$.
\end{proof}

We are now able to prove Theorem \ref{thm:main}.
\begin{proof}[Proof of Theorem 3.1] We substitute \eqref{poopinequality} into \eqref{eq:cesaromeans} to obtain
\begin{align*}
	&\intim \intbd \abs{ K_n^\delta ( \xb, \eb, \yb, \yb') }
	w(\alpha, \beta, \mu; \xb, \yb) \, d\yb\, d\xb\\
	\leq & c \intim \inti \left| \sum_{j=0}^n c_{k,n}^\delta
	\sum_{k=0}^j \frac{ k +\mu + \frac{ d-1 }{2} }{ \mu + \frac{d-1}{2} }
	C_k^{ \left( \mu + \frac{ d-1 }{2} \right) } \left( u \right) \right. \\
	& \left. \qquad \qquad \times 
	\left[ \PP_{j-k}^{ \left( \alpha, \beta \right) } \left( \xb \right) \right] ^T 
	\PP_{j-k}^{ \left( \alpha, \beta \right) } \left( \eb \right) \right|
	\, \left( 1- u^2 \right)^{ \frac{ d-2 }{2} + \mu } \, du
	\, w^{ \left( \alpha, \beta \right) } (\xb) \, d\xb.
\end{align*}
After substituting in the well known identity for Gegenbauer polynomials, 
\begin{displaymath}
	\displaystyle \frac{n+\lambda}{\lambda} C_n^\lambda(x)=\tilde{C}_n^\lambda(1)\tilde{C}_n^\lambda(x),
\end{displaymath}
we arrive at the following inequality.
\begin{align}\label{al:unifbdsum}
	& \intim \intbd \abs{ K_n^\delta ( \xb, \eb, \yb, \yb') }
	w(\alpha,\beta,\mu; \xb, \yb)\\ 
	\leq & c \intim \inti \left| \sum_{j=0}^n c_{k,n}^\delta
	\sum_{k=0}^j \widetilde{C}_k^{ \left( \mu + \frac{ d-1 }{2} \right) } 
	\left( u \right) \right. 
	\widetilde{C}_k^{ \left( \mu + \frac{ d-1 }{2} \right) } 
	\left( 1 \right) \notag\\
	& \left. \qquad \qquad \times 
	\left[ \PP_{j-k}^{ \left( \alpha, \beta \right) } \left( \xb \right) \right] ^T 
	\PP_{j-k}^{ \left( \alpha, \beta \right) } \left( \eb \right) \right|
	\, \left( 1- u^2 \right)^{ \frac{ d-2 }{2} + \mu } \, du
	\, w^{ \left( \alpha, \beta \right) } (\xb) \, d\xb. \notag
\end{align}
Since the Gegenbauer polynomials are a subset of the Jacobi polynomials, the proof of our theorem then follows from the results on the cube; specifically, Theorem 1.1 and Lemma 2.2 from \cite{LiXu}, which shows that the expression in \eqref{al:unifbdsum} is uniformly bounded for our choice of $\delta$.
\end{proof}

\section{Further Investigation}
It should be noted that it is not known whether the bound for $\delta$ obtained in Theorem \ref{thm:main} is sharp.  In fact, the sharpness of the bound for $\delta$ on the cube in \cite{LiXu}, are, to the author's knowledge, not known.  These questions are areas for further investigation.

\setcounter{equation}{0}


\begin{thebibliography}{00}
\bibitem{CastellFilbirXu}
	{\sc {W. zu Castell, F. Filbir, Y. Xu}},
	\textit{Ces\`aro means of Jacobi expansions on the parabolic biangle},
	J. Approx. Theory, \textbf{159} (2009), 167-179.

\bibitem{DaiXu}
	{\sc {F. Dai, Y. Xu}},
	\textit{Boundedness of projection operators and Ces\`aro means in weighted $L^p$ space on the unit sphere},
	Trans. Amer. Math. Soc., \textbf{361} (2009), 3189-3221.
	
\bibitem{Marchi}
	{\sc {S. De Marchi, M. Marchioro, A. Sommariva}},
	\textit{Polynomial approximation and cubature at approximate Fekete and Leja points of the cylinder},
	Appl. Math. Comput., \textbf{218} (2012), 10617-10629.

\bibitem{DunklXu}
	{\sc {C. F. Dunkl, Y. Xu}},
	\textit{Orthogonal Polynomials of Several Variables},
	Encyclopedia of Mathematics and its Applications, vol. \textbf{81}, Cambridge University Press, 2001.

\bibitem{Gasper}
	{\sc {G. Gasper}},
	\textit{Banach Algebras for Jacobi Series and Positivity of a Kernel},
	Ann. of Math., \textbf{95} (1972), 261-280.

\bibitem{LiXu}
	{\sc {Z. Li, Y. Xu}},
	\textit{Summability of Product Jacobi Expansions},
	J. Approx. Theory, \textbf{104} (2000), 287-301. 

\bibitem{Muller}
	{\sc {C. M\"uller}},
	\textit{Spherical Harmonics},
	Lecture Notes in Mathematics, vol. \textbf{17}, Springer-Verlag, 1966.

\bibitem{Szego}
	{\sc {G. Szeg\"o}},
	\textit{Orthogonal Polynomials},
	Colloquium Publications, vol. \textbf{23}, American Mathematical Society, 2000.

\bibitem{Wade}
	{\sc {J. Wade}},
	\textit{A discretized Fourier orthogonal expansion in orthogonal polynomials on a cylinder},
	J. Approx. Theory, \textbf{162} (2010), 1545-1576.

\bibitem{Xu}
	{\sc {Y. Xu} },
	\textit{Summability of Fourier Orthogonal Series for Jacobi Weight on a Ball in $\RR^d$},
	Trans. Amer. Math. Soc., \textbf{351} (1999), 2439-2458.

\bibitem{Xu2}
	{\sc {Y. Xu} },
	\textit{Fourier Series and Approximation on Hexagonal and Triangular Domains},
	Constr. Approx., \textbf{31} (2010), 115-138.
\end{thebibliography}
\end{document}